\pdfoutput=1
\documentclass{amsart}

\usepackage{color}
\usepackage{graphicx}
\usepackage[mathcal,mathscr]{euscript}
\usepackage{bbm}
\usepackage{float}
\usepackage{amsmath,amsthm,amssymb,amsfonts,amscd,epsfig,latexsym,graphicx,textcomp}
\usepackage{tikz-cd}
\usepackage{psfrag}
\usepackage[all]{xy}
\usepackage{stackengine}
\usepackage{romannum}
\usepackage{comment}
\usepackage{hhline}
\usepackage{diagbox}
\usepackage{float}
\usepackage{caption}
\usepackage{eufrak}
\usepackage{diagbox}
\usepackage{makecell}
\usepackage{slashed}
\usepackage{MnSymbol}

\restylefloat{figure}

\newtheorem*{theorem*}{Main Theorem}

\usepackage{thmtools}
\usepackage{thm-restate}

\usepackage{hyperref}

\usepackage{cleveref}

\newtheorem{theorem}{Theorem}[section]
\newtheorem{lemma}[theorem]{Lemma}

\theoremstyle{definition}
\newtheorem{definition}[theorem]{Definition}

\theoremstyle{remark}

\numberwithin{equation}{section}

\newcommand{\PMEdgeDiag}{\raisebox{-0.33\height}{\includegraphics[scale=0.25]{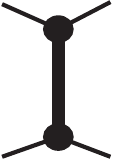}}}
\newcommand{\IIDiag}{\raisebox{-0.33\height}{\includegraphics[scale=0.25]{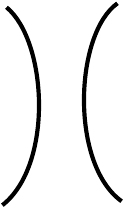}}}

\newcommand{\ZZ}{{\mathbb Z}}

\newcommand{\generalvertexbracketvertex}{\raisebox{-0.33\height}{\includegraphics[scale=0.25]{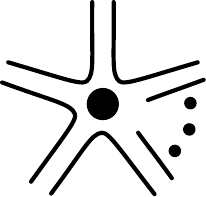}}}
\newcommand{\generalvertexbracketzero}{\raisebox{-0.33\height}{\includegraphics[scale=0.25]{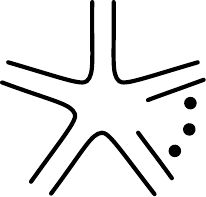}}}
\newcommand{\generalvertexbracketone}{\raisebox{-0.33\height}{\includegraphics[scale=0.25]{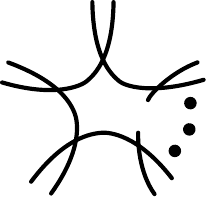}}}

\newcommand{\vertexbracketvertex}{\raisebox{-0.33\height}{\includegraphics[scale=1.2]{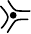}}}
\newcommand{\vertexbracketzero}{\raisebox{-0.33\height}{\includegraphics[scale=1.2]{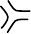}}}
\newcommand{\vertexbracketone}{\raisebox{-0.33\height}{\includegraphics[scale=1.2]{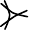}}}
 
\newcommand{\PentagonDot}{\raisebox{-0.4 \height}{\includegraphics[scale=0.25]{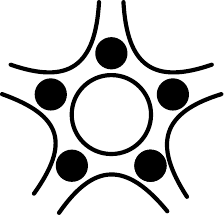}}}

\newcommand{\PentBDDot}{\raisebox{-0.4 \height}{\includegraphics[scale=0.25]{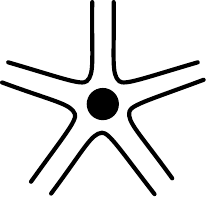}}}
\newcommand{\VStickman}{\raisebox{-0.4 \height}{\includegraphics[scale=0.25]{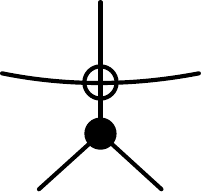}}}
\newcommand{\VStickmanDot}{\raisebox{-0.4 \height}{\includegraphics[scale=0.25]{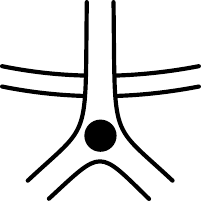}}}

\newcommand{\GrimaceDot}{\raisebox{-0.4 \height}{\includegraphics[scale=0.25]{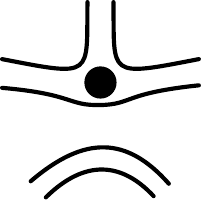}}}

\newcommand{\QuadDot}{\raisebox{-0.4 \height}{\includegraphics[scale=0.25]{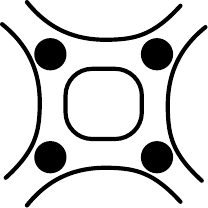}}}

\newcommand{\QuadBDDot}{\raisebox{-0.4 \height}{\includegraphics[scale=0.25]{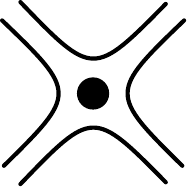}}}

\newcommand{\QuadVirt}{\raisebox{-0.4 \height}{\includegraphics[scale=0.25]{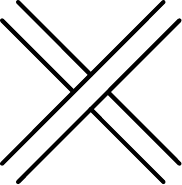}}}

\newcommand{\QuadBDZero}{\raisebox{-0.4 \height}{\includegraphics[scale=0.25]{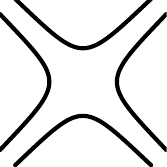}}}
\newcommand{\QuadBDOne}{\raisebox{-0.4 \height}{\includegraphics[scale=0.25]{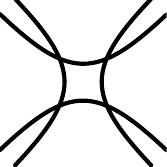}}}

\newcommand{\II}{\raisebox{-0.4 \height}{\includegraphics[scale=0.25]{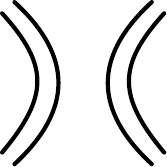}}}
\newcommand{\Equal}{\raisebox{-0.4 \height}{\includegraphics[scale=0.25]{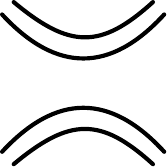}}}

\newcommand{\BubbleDot}{\raisebox{-0.4 \height}{\includegraphics[scale=0.25]{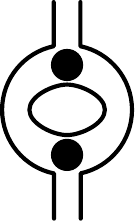}}}
\newcommand{\StraightEdge}{\raisebox{-0.4 \height}{\includegraphics[scale=0.25]{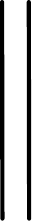}}}

\newcommand{\TriangleDot}{\raisebox{-0.4 \height}{\includegraphics[scale=0.25]{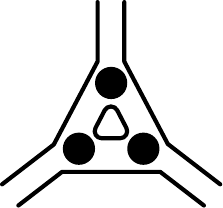}}}
\newcommand{\TriangleBDDot}{\raisebox{-0.4 \height}{\includegraphics[scale=0.25]{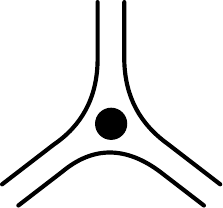}}}

\newcommand{\HeartGraph}{\raisebox{-0.4 \height}{\includegraphics[scale=0.25]{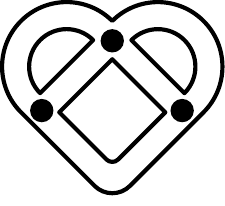}}}

\begin{document}

\title{New relations for the vertex polynomial}

\thanks{}

\author{Scott Baldridge}
\address{Department of Mathematics, Louisiana State University,
Baton Rouge, LA}
\email{baldridge@math.lsu.edu}

\author{Ben McCarty}
\address{Department of Mathematical Sciences, University of Memphis,
Memphis, TN}
\email{ben.mccarty@memphis.edu}

\date{}

\begin{abstract} 
We extend the vertex polynomial to graphs of arbitrary degree and prove local relations that hold when a graph contains a digon, triangle, quadrilateral or pentagon.
\end{abstract}

\maketitle

\pagenumbering{arabic}

\section{Introduction}
Inspired by one of Penrose's several formulas for counting $3$-edge colorings \cite{Penrose}, the vertex polynomial was introduced for trivalent graphs in \cite{BM-Quantum} using the following relations:
\begin{eqnarray*}
V\left( \vertexbracketvertex \right) &=&  V\left( \vertexbracketzero \right) \ - \ V\left(\vertexbracketone \right), \mbox{\ and} \\ [.2cm]
V\left( \bigcirc  \right) & = & n.
\end{eqnarray*}
When evaluated at $n=2$ it recovers (a multiple of) the number of Tait colorings of a bridgeless planar cubic graph.  Consequently, non-vanishing of $V(G,2)$ for every such graph is equivalent to the Four Color Theorem.  

The same state-sum construction extends to graphs of arbitrary degree.  The reason is that the zero resolutions \vertexbracketzero \ and one resolutions \vertexbracketone, make sense regardless of degree.  For example, a degree four vertex can be resolved as follows:
$$V\left( \QuadBDDot \right) =  V\left( \QuadBDZero \right) \ - \ V\left( \QuadBDOne \right).\ $$
The general case is handled similarly and defined for the first time in this paper (see \Cref{defn:VertexPenrose}).  


While the behavior of the vertex polynomial of a trivalent graph with respect to blowups was described in \cite{BM-Quantum},  no useful relations were known for digons, quadrilaterals or pentagons.  \Cref{thm:MainTheorem} supplies them, continuing in the spirit of what \cite{BM-Penrose} did for the Penrose polynomial.  The digon and the triangle relations are relatively elementary, but the quadrilateral and pentagon are more interesting. The extension of the vertex polynomial, though immediate, enables one to produce the remaining relations of \Cref{thm:MainTheorem}.  Collectively, the relations mirror the classical reducibility program for the Four Color Theorem:  every planar trivalent graph contains a face of size at most five, and so local rewriting rules for those faces are precisely the tools needed to reduce a putative minimal counterexample.

\section{Definitions}\label{sec:Defn}
We briefly recall the needed definitions here and refer the reader to \cite{BM-Quantum, BM-Color, BKR} for further details.  Given an abstract graph $G(V,E)$, a ribbon graph generalizes the notion of a plane graph in that it is an embedding $i:G\rightarrow \Gamma$ where $G$ is thought of as a $1$-dimensional CW complex and $\Gamma$ is a surface with boundary that deformation retracts onto $i(G)$.  Ribbon graphs will be represented by diagrams in the plane in which the cyclic ordering of the edges at each vertex determines an embedding.  

A perfect matching of a graph $G(V,E)$ is a subset of the non-loop edges of the graph, $M\subset E$ such that each vertex is incident to exactly one edge in the subset.  The term \emph{perfect matching graph} will refer to a choice of ribbon graph $\Gamma$ for $G$ and a perfect matching $M$.  It will be notated as $\Gamma_M$.  While there will typically be many perfect matching graphs for a given abstract graph, one may construct the \emph{blowup} of a graph, denoted $\Gamma^\flat$ by replacing every degree $r$ vertex of $\Gamma$ with an $r$-cycle (see \Cref{fig:blowup-and-vertex}).  The blowup has a canonical perfect matching given by the original edges of $\Gamma$.  We denote the canonical perfect matching graph obtained in this way by $\Gamma_E^\flat$.

\begin{figure}
\psfragscanon
\psfrag{b}{$\flat$}\psfrag{=}{$=$}
\psfrag{G}{$\Gamma_\bullet  \ = $}
\includegraphics[scale=.6]{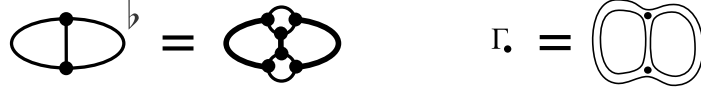}
\caption{The blowup and vertex ribbon diagram of the theta graph.} \label{fig:blowup-and-vertex}
\end{figure}

For our purposes in this paper, the {\em all-zero state} will be the result of replacing each edge $\PMEdgeDiag$ in $\Gamma^\flat_E$ with a  $\IIDiag$.  Importantly, this means that the circles of the all-zero state represent the boundaries of the faces of the original ribbon graph.  The {\em vertex ribbon diagram}, denoted $\Gamma_\bullet$, is the all-zero state of $\Gamma^\flat_E$ together with dots  placed where each vertex of the graph was before taking the all-zero resolution, i.e., a $\bullet$ is placed in the region where the three face(s) are incident to the vertex (see the right-hand side  of \Cref{fig:blowup-and-vertex}).  The inclusion of the dot allows one to tell, at a glance, whether the vertex has been resolved (into a zero or one resolution) or not.

We are now ready to define the \emph{vertex polynomial.}


\begin{definition}\label{defn:VertexPenrose}
Let $G(V,E)$ be an abstract graph with ribbon diagram $\Gamma$.  Let $\Gamma_\bullet$ be the vertex ribbon diagram of $\Gamma$.  The {\em vertex polynomial}, $V(\Gamma,n)$, is characterized by applying the following rules to the vertex ribbon diagram $\Gamma_\bullet$ for $n\in\ZZ$:
\begin{eqnarray}
V\left( \generalvertexbracketvertex \right) &=&  V\left( \generalvertexbracketzero \right) \ - \ V\left(\generalvertexbracketone \right) \label{eq:vertexP-bracket}\\ [.2cm]
\label{eq:VPimmersed_circle}
V\left( \bigcirc  \right)& = & n\\[.2cm]
V(\Gamma_1 \sqcup \Gamma_2)&=& V(\Gamma_1) \cdot V( \Gamma_2)\label{eq:vertexP-disjoint-union}
\end{eqnarray}
\end{definition}

While the polynomial in \Cref{defn:VertexPenrose} is defined recursively, it is often helpful to use the \emph{hypercube of vertex states} for computations.  We refer to the replacement of \generalvertexbracketvertex \ with \generalvertexbracketzero \ as a zero resolution and its replacement with \generalvertexbracketone \ as a one resolution.  The all-zero resolution is on the left, and column $i$ contains $i$ one resolutions.  An arrow is drawn between states that differ in a single resolution (see \Cref{fig:vertex-state-ex}).  

\begin{figure}
\includegraphics[scale=.4]{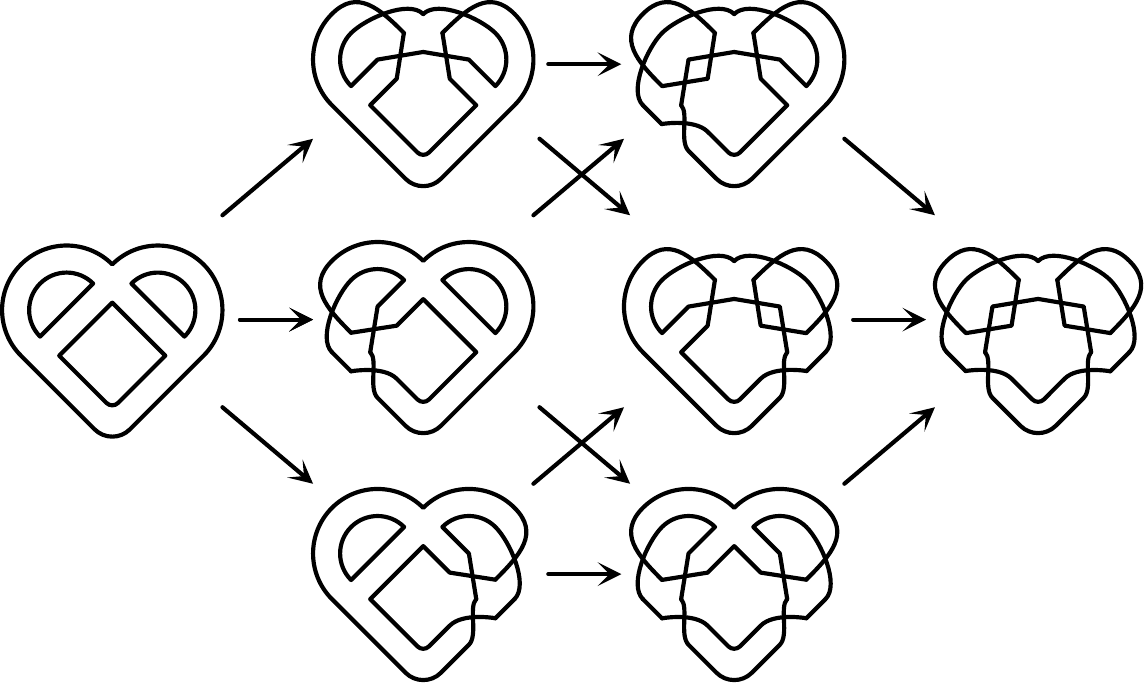}
\caption{A hypercube of vertex states.}\label{fig:vertex-state-ex}
\end{figure}

Each state of the hypercube of vertex states represents a choice of a zero or one resolution at every vertex, and it facilitates computation in that one may sum over the $2^{|V|}$ states with each state contributing $(-1)^i n^k$ where $i$ is the column in which the state resides, and $k$ is the number of immersed circles in the state.  For example, for the graph whose hypercube of vertex states is shown in \Cref{fig:vertex-state-ex} we have 
$$V\left(\HeartGraph\right) = n^4 - 3n^2+3n^2 -n^4 = 0.$$
While the polynomial is zero for this example, it is not always.  For example, the graph in \Cref{fig:blowup-and-vertex} has vertex polynomial $2n(n^2-1)$.  

\section{Main Results}
The proof of the relations comes down to a careful study of the hypercube of vertex resolutions.  The following lemma will be useful in that process in that it allows many of the states to be ignored due to the fact that their contribution to the polynomial will be canceled by other states.  

\begin{lemma}
States containing any of the configurations shown in \Cref{fig:LolliSymmetry}, where the dotted path does not contain any of the other arcs, and with the free ends joined in any manner, come in cancelling pairs, and can therefore be ignored when computing the vertex polynomial.
\label{lem:LolliDies}
\end{lemma}

\begin{figure}[H]
\includegraphics[scale=.5]{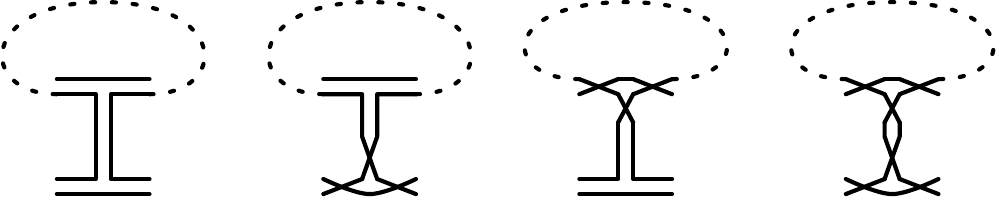}
\caption{States containing this configuration do not contribute to the vertex polynomial.}
\label{fig:LolliSymmetry}
\end{figure}

\begin{proof}
The two configurations on the left side of \Cref{fig:LolliSymmetry} form a cancelling pair:  joining the the free ends with arcs in any manner results in two configurations with the same number of circles, and different parity in the number of vertex one resolutions. The argument is similar for the two configurations on the right.
\end{proof}

We now come to the main result of this paper:  a set of relations for the digon, triangle, quadrilateral and the pentagon.

\begin{theorem}
\label{thm:MainTheorem}
Given that the edges emanating from each configuration on the left are unique, and $R^i$ represents a rotation of the configuration by $2\pi i/5$ radians, the vertex polynomial satisfies the following relations:
\begin{enumerate}
\item $V\left( \BubbleDot \right) =  2n V\left( \StraightEdge \right),$

\item $V\left( \TriangleDot \right) =  n V\left( \TriangleBDDot \right),$

\item $V\left( \QuadDot \right) =  n V\left( \QuadBDZero \right) \ + \ n V\left( \QuadBDOne \right)\ + \ 2 V\left(\II \right)\ + \ 2 V\left( \Equal \right) + 2 V\left(\QuadVirt\right),$

\item $V\left( \PentagonDot \right) =  n V\left( \PentBDDot \right) \ + \ \sum_{i=0}^4 V\left(R^i \left(\VStickmanDot \right) \right)\ + \ \sum_{i=0}^4 V\left(R^i \left(\GrimaceDot \right) \right).$
\end{enumerate}

\end{theorem}

\begin{proof}
The second relation was already proven in \cite{BM-Quantum}.  For the first relation, any state that has a zero smoothing on top, and a one smoothing on the bottom, or vice-versa, is one of the configurations in \Cref{lem:LolliDies}, and therefore contributes zero to the vertex polynomial.  The remaining states with two zero smoothings on the configuration or two one smoothings are topologically identical, and match a single band, with an extra circle, which proves the relation.

The third and fourth relations are handled in a similar manner:  from the cube of resolutions, remove any state where \Cref{lem:LolliDies} applies. What remains is a collection of states that naturally pair up to correspond to the configurations on the right, sometimes with an extra circle, which comes out as a factor of $n$.  For example, the two states at the top of \Cref{fig:PentExample} appear in the hypercube of states for the pentagon, with different parity (i.e. sign).  The circle data is clearly identical to the configurations shown at the bottom, as is the parity. The sub-hypercubes associated with these two pictures can be replaced with $R^4 \left( \VStickmanDot \right)$, as shown in the relation.  The other pairings produce the remaining terms on the right-hand side of the relation.

\begin{figure}
\includegraphics[scale=.4]{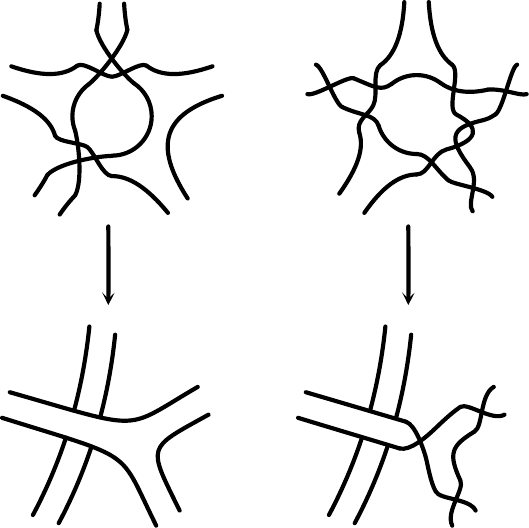}
\caption{Two states for the pentagon configuration, and their associated configurations.}
\label{fig:PentExample}
\end{figure}

\end{proof}

\section{Conclusion}

For trivalent graphs it was shown in \cite{BM-Quantum} (and it really goes back to \cite{Penrose}) that positivity of the vertex polynomial at $n = 2$ for every bridgeless planar cubic graph is equivalent to the Four Color Theorem.  Going back to Kempe's own failed proof we already know that a minimal counterexample cannot contain a quadrilateral face (or smaller).  Our pentagon relation in \Cref{thm:MainTheorem} might look promising initially since it rewrites the vertex polynomial of a trivalent graph containing a pentagon in terms of smaller configurations.   However, the terms involving the five rotated nonplanar configurations \VStickman \ may have vertex polynomials that are negative, zero or positive; the relation only asserts that the sum equals the left-hand side. It is therefore still possible for both sides to vanish. This is the precise combinatorial obstruction that remains—the same obstruction that has historically blocked reducibility arguments at the pentagon.

Whether the non-planar summands that appear can be controlled, or whether a different set of planar-preserving pentagon identities exists, is left open.



\begin{thebibliography}{99}



\bibitem{BKR} S. Baldridge, L. Kauffman, and W. Rushworth, {\em On ribbon graphs and virtual links}, European Journal of Combinatorics {\bf 103}, June 2022, doi: 10.1016/j.ejc.2022.103520, arXiv: 2010.04238.


\bibitem{BM-Color} S. Baldridge and B. McCarty, {\em A topological quantum field theory approach to graph coloring}, arXiv:2303.12010.


\bibitem{BM-Quantum} S. Baldridge and B. McCarty, {\em Quantum state systems that count perfect matchings}, arXiv:2401.07939.

\bibitem{BM-Penrose} S. Baldridge and B. McCarty, {\em New relations for the Penrose polynomial}, arXiv:2606.06643.












\bibitem{Penrose} R. Penrose, ``Applications of negative dimensional tensors,'' in {\em Combinatorial Mathematics and Its Applications}, Academic Press (1971).




\end{thebibliography}
\end{document}